\documentclass[12pt]{article}

\usepackage[top=2.4cm,bottom=2.2cm,left=2.6cm,right=2cm]{geometry}
\usepackage{latexsym,epsfig,psfrag}
\usepackage{eepic,colordvi,amscd}

\usepackage{amsmath,amsthm,dsfont,amsfonts,amssymb,fancyhdr}
\usepackage{enumerate}
\usepackage[usenames,dvipsnames]{color}
\usepackage{bbm,soul,supertabular,longtable,verbatim,extarrows}
\usepackage{titlesec}
\usepackage{graphicx}
\usepackage{BOONDOX-cal}
\usepackage{cite}
\usepackage{tikz}
\usepackage{booktabs,multirow,makecell}
\usepackage{authblk}
\usepackage{color}
\usetikzlibrary[trees]
\usepackage {url}

\newcommand{\gskip}{\vspace{12pt}}

\newcommand{\RR}{{\mathbb R}}
\usepackage{lipsum}
\usepackage{mathrsfs}
\usepackage{indentfirst}
\usepackage[colorlinks=true, allcolors=blue]{hyperref}
\newtheorem{thm}{Theorem}[section]
\newtheorem{cor}[thm]{Corollary}
\newtheorem{lem}[thm]{Lemma}

\newtheorem{prop}[thm]{Proposition}
\theoremstyle{definition}

\theoremstyle{remark}

\numberwithin{equation}{subsection}
\newcounter{si}
\addtocounter{si}{1}

\usepackage{mathtools}

\newcommand{\Rmnum}[1]{\uppercase\expandafter{\romannumeral #1}}  
\newcommand{\GG}{\Gamma}

\newcommand{\SSII}{(\thesi) \addtocounter{si}{1}}

\begin{document}
\title{Non-geometric distance-regular graphs of diameter at least $3$ with smallest eigenvalue at least $-3$}

\author{
Jack H. Koolen$^{\,\rm 1, \, \rm 2}$,
Kefan Yu$^{\,\rm 3}$,
Xiaoye Liang\thanks{Corresponding author.} $^{\,\rm 4,\, \rm 5}$, \\
Harrison Choi$^{\,\rm 6}$, and
Greg Markowsky$^{\,\rm 6}$\\

 {\small {\tt koolen@ustc.edu.cn}~~
         {\tt ykf97@mail.ustc.edu.cn}~~
         {\tt liangxy0105@foxmail.com}\\
         {\tt hcho0043@student.monash.edu}~~
         {\tt gmarkowsky@gmail.com} }\\
         
  {\footnotesize{$^{\rm 1}$School of Mathematical Sciences, University of Science and Technology of China, Hefei, PR China}}\\
  {\footnotesize{$^{\rm 2}$Wen-Tsun Wu Key Laboratory of the CAS, University of Science and Technology of China,  Hefei, PR China}}\\
  {\footnotesize{$^{\rm 3}$ School of Computer Science and Technology, University of Science and Technology of China, Hefei, PR China}}\\
  {\footnotesize{$^{\rm 4}$School of Mathematics and Physics,  Anhui Jianzhu University, Hefei, PR China}}\\
   {\footnotesize{$^{\rm 5}$Operations Research and Data Science Laboratory, Anhui Jianzhu University, Hefei, PR China}}\\
   {\footnotesize{$^{\rm 6}$Department of Mathematics,  Monash University, Australia}}\\
 }
 \date{}
\maketitle

\begin{abstract}
	In this paper, we classify non-geometric distance-regular graphs of diameter at least $3$ with smallest eigenvalue at least $-3$. This is progress towards what is hoped to be an eventual complete classification of distance-regular graphs with smallest eigenvalue at least $-3$, analogous to existing classification results available in the case that the smallest eigenvalue is at least $-2$.
\end{abstract}

\section{Introduction} \label{intro}

A great deal of information is known about regular graphs with smallest eigenvalue at least $ -2 $. For instance, it was proved in \cite{CameronEtAl} that if a regular graph $\Gamma$ with more than $ 28 $ vertices
is connected and has smallest eigenvalue at least $-2$, then it is either a line graph or a cocktail party graph (see also \cite[Theorem 3.12.2]{BCN}). Furthermore, the strongly regular graphs with smallest eigenvalue (at least) $ -2 $ have been completely classified, and form a short list (cf. \cite[Theorem 3.12.4]{BCN}). It is natural to attempt to extend these results and to similarly understand regular graphs with smallest eigenvalue at least $ -3 $. As one would expect, this is much harder than the $ -2 $ case, but in recent years a number of researchers have made significant contributions to this project. The purpose of this paper is to expand knowledge in this area, and in particular to understand better non-geometric distance-regular graphs with smallest eigenvalue at least $-3$.

The geometric distance-regular graphs with $c_2 \geq  2$ and smallest eigenvalue $-3$ have been completely classified in \cite{gdrg-3}, \cite{non-exist}, and \cite{45}. Classifying all geometric distance-regular graphs with $c_2 = 1$ is still an open problem, and it is probably very hard. Our focus then is on the non-geometric case.

We will consider when the  diameter is at least $3$ and the smallest eigenvalue  is at least $-3$ in this paper. Our main result is the following theorem.

\begin{thm}\label{main}
 Let $\Gamma$ be a non-geometric distance-regular graph with diameter at least 3. 
 If  the smallest eigenvalue of $\Gamma $ is at least $-3$,  then one of the following cases holds:
\begin{enumerate}[(a)]
  \item $\Gamma$ is the Odd graph $O_4$ with $\iota(\Gamma)=\{4,3,3;1,1,2\}$;
  \item $\Gamma$ is the Sylvester graph with $\iota(\Gamma)=\{5,4,2;1,1,4\}$;
  \item $\Gamma$ is the second subconstituent of Hoffman-Singleton graph with $\iota(\Gamma)=\{6,5,1;1,1,6\}$;
  \item $\Gamma$ is the Perkel graph with $\iota(\Gamma)=\{6,5,2;1,1,3\}$;
  \item $\Gamma$ is the symplectic $7$-cover of $K_9$ with $\iota(\Gamma)=\{8,6,1;1,1,8\}$;

  \item $\Gamma$ is the Coxeter graph $\iota(\Gamma)=\{3,2,2,1;1,1,1,2\}$;
  \item $\Gamma$ is the dodecahedron with $\iota(\Gamma)=\{3,2,1,1,1;1,1,1,2,3\}$;
  \item $\Gamma$ is the Biggs-Smith graph with $\iota(\Gamma)=\{3,2,2,2,1,1,1;1,1,1,1,1,1,3\}$;
 \item $\Gamma$ is the Wells graph with $\iota(\Gamma)=\{5,4,1,1;1,1,4,5\}$;

  \item $\Gamma$ is the icosahedron with $\iota(\Gamma)=\{5,2,1;1,2,5\}$;		
  \item $\Gamma$ is the Doro graph with $\iota(\Gamma)=\{10,6,4;1,2,5\}$;

  \item $\Gamma$ is the halved $6$-cube with $\iota(\Gamma)=\{15,6,1;1,6,15\}$;
  \item $\Gamma$ is the Gosset graph with $\iota(\Gamma)=\{27,10,1;1,10,27\}$;

  \item $\Gamma$ is the halved $7$-cube with $\iota(\Gamma)=\{21,10,3;1,6,15\}$;

  \item $\Gamma$ is the Klein graph with $\iota(\Gamma)=\{7,4,1;1,2,7\}$;
  \item $\Gamma$ is a distance-regular graph with $\iota(\Gamma)=\{9,6,1;1,2,9\}$;
  \item $\Gamma$ is the Doob graph with $\iota(\Gamma)=\{9,6,3;1,2,3\}$;
  \item $\Gamma$ is a  distance-regular graph with $\iota(\Gamma)=\{15,10,1;1,2,15\}$;
   \item $\Gamma$ is a putative distance-regular graph with $\iota(\Gamma)=\{18,12,1;1,2,18\}$.
 \end{enumerate}
\end{thm}

{\bf Remark:} When $\iota(\Gamma)=\{9,6,1;1,2,9\}$, by  \cite{spence2000}  we know that there exist exactly two non-geometric examples (up to isomorphism). When $\iota(\Gamma)=\{15,10,1;1,2,15\}$, there also exist non-geometric examples by  \cite{Brouwer2003}. Section \ref{sec:compD3} below contains a detailed discussion of these two (and other) cases.

{\bf Remark:}The case $D=2$, the strongly regular case, requires (as usual) very different techniques and is not addressed in this paper. The situation is discussed in the recent paper \cite{geb}, and the interested reader is referred there.

This paper is devoted to the proof and discussion for this result. The situation is fairly well understood when  $\theta_{\min} \geq -2$, so we will generally assume $-3 \leq \theta_{\min} < -2$. All distance-regular graphs are known when $a_1$ is large in relation to $k$, in particular when $k \leq 2a_1 +2$  (see \cite{koo-park}).  If $k \leq 2a_1 +2$, then we have a Taylor graph, a line graph, the Johnson graph $J(7,3)$ or the halved 7-cube.  A line graph has smallest eigenvalue at least $-2$  and the Johnson graph $J(7,3)$ is geometric.  So if $k \leq  2a_1 + 2$ we only have to consider the Taylor graphs and the halved 7-cube.

The rest of the paper is organized as follows. In Sections  \ref{prelim} and \ref{lin_alg_prelim}, we will give some definitions and preliminary results.  In Section \ref{c_2=1}, we consider the case when $c_2=1$; as is often the case in this field, this requires different tools than the case $c_2 > 1$. In Section \ref{sec:Taylor} we give the classification when $\Gamma$ is a Taylor graph. In Section \ref{sec:terwgraphs} we give the classification when $\Gamma$ is a
Terwilliger graph. In Section \ref{c2geq2} we determine a diameter bound and  give a  bound for $c_2$ when $\Gamma$ is non-geometric and $-3\leq \theta_{\min}<-2, c_2\geq 2, k \geq 2a_1+3$. In Section \ref{sec:a1bound}, in order to search by computer, we use the bounds in Section \ref{c2geq2} to give an upper bound on $a_1$. In Section \ref{sec:comput} we present the resulting computational results. In Section \ref{sec:proof} we give the proof of the main theorem.

\if2
\section{Outline of paper} \label{outline}

Because there are a number of separate cases considered, we will give an overview of the paper in this section. We are interested in the classification of all non-geometric distance-regular graphs of diameter at least 3 having smallest eigenvalue -3 or larger. The situation is fairly well understood when $\theta_{\min} \geq -2$, so we will generally assume $-3 \leq \theta_{\min} < -2$. All distance-regular graphs are known when $a_1$ is large in relation to $k$, in particular when $k \leq 2a_1 +2$ (see \cite{KP2}). We will therefore assume that this is not the case, namely that $k \geq 2a_1 +2$. In this case, our graph must be a Taylor graph, a line graph, the Johnson graph $J(7,3)$, or the halved $7$-cube.

\gskip

The following two sections, Sections \ref{prelim} and \ref{lin_alg_prelim}, give the necessary preliminary combinatorial and linear algebraic facts, respectively. These sections also contain the required definitions.

\gskip

The remaining sections comprise the proof of Theorem \ref{main}. As is often the case in this field, different methods are required depending on whether $c_2=1$ or $c_2 \geq 2 $. The $c_2=1$ case is easier for this problem, and we will find all relevant graphs in Section \ref{c_2=1}.
\fi

\section{Combinatorial preliminaries} \label{prelim}

All the graphs considered in this paper are finite, undirected and
simple (for unexplained terminology and more details, see for example \cite{BCN} or \cite{drgsurvey}). Let $\Gamma$ be a
connected graph and let $V(\Gamma)$ be the vertex set of $\Gamma$. The {\it distance} $d(x,y)$ between
any two vertices $x,y$ of $\Gamma$
is the length of a shortest path between $x$ and $y$ in $\Gamma$. The {\it diameter} of $\Gamma$ is the maximum distance
occurring in $\Gamma$ and we will denote this by $D = D(\Gamma)$. For a vertex $x \in V(\Gamma)$, define $\Gamma_i(x)$ to be the set of
vertices which are at distance $i$ from $x~(0\le i\le
D)$, and when the choice of $x$ is unimportant we will simply write $\Gamma_i$. In addition, define $\Gamma_{-1}(x)=\Gamma_{D+1}(x)
= \emptyset$. We write $x\sim_{\Gamma} y$ or simply $x\sim y$ if two vertices $x$ and $y$ are adjacent in $\Gamma$. A connected graph $\Gamma$ with diameter $D$ is called
{\em distance-regular} if there are integers $b_i,c_i$ $(0 \le i
\leq D)$ such that for any two vertices $x,y \in V(\Gamma)$ with $d(x,y)=i$, there are precisely $c_i$
neighbors of $y$ in
$\Gamma_{i-1}(x)$ and $b_i$ neighbors of $y$ in $\Gamma_{i+1}(x)$
(cf. \cite[page 126]{BCN}). In particular, a distance-regular graph $\Gamma$ is regular with valency
$k := b_0$ and we define $a_i:=k-b_i-c_i$ for notational convenience.
The numbers $a_i$, $b_{i}$ and $c_i~(0\leq i\leq D)$ are called the {\it
intersection numbers} of $\Gamma$, and the sequence $\{b_0,b_1, \ldots, b_{D-1};c_1,c_2, \ldots , c_D\}$ is the {\it intersection array} of $\Gamma$, we will denote it by $\iota(\Gamma)$. Note that always $c_0=b_D=0$, $b_0 = k$ and $c_1=1$.

There are two other related classes of graphs worth mentioning here. The first are the {\it strongly regular} graphs, which are simply the distance-regular graphs of diameter $2$. We say that a strongly regular graph $\Gamma$ has parameters $(v, k, \lambda, \mu)$, where $v:=|V(\Gamma)|$, $\lambda := a_1$, and $\mu := c_2$. The theory of these graphs is significantly different than that of higher diameter distance-regular graphs, and for this reason they are often identified separately.

The intersection numbers of a distance-regular graph $\Gamma$ with diameter $D$ and valency $k$ satisfy
(cf. \cite[Proposition 4.1.6]{BCN})\\

\noindent (i) $k=b_0> b_1\geq \cdots \geq b_{D-1}$;\\
(ii) $1=c_1\leq c_2\leq \cdots \leq c_{D}$;\\
(iii) $b_i\ge c_j$ \mbox{ if }$i+j\le D$.\\

Moreover, if we fix a vertex $x$ of $\Gamma$, then $k_i(x):=|\Gamma_i(x)|$ does not depend on the
choice of $x$ as $c_{i} k_{i}(x) =
b_{i-1} k_{i-1}(x)$ holds for $i =1, 2, \cdots, D$, and thus $k_i:=k_i(x) = \frac{b_0 b_1 \cdots b_{i-1}}{c_1 c_2 \cdots  c_{i}}$. Define $k_0=1$ and $k:=k_1$. So $v=|V(\Gamma)|=1+k+k_2+\cdots+k_D$.

By the \emph{eigenvalues}  of a graph $\Gamma$, we mean the eigenvalues  of its adjacency matrix~$A:=A(\Gamma)$. If the $\Gamma$ is distance-regular of diameter $D$, its adjacency matrix $A$ has always exactly $D+1$ distinct eigenvalues, and we will write $k=\theta_0 > \theta_1 > \cdots > \theta_D$ to describe these distinct eigenvalues, and refer to $\theta_0, \ldots , \theta_D$ as simply the {\it eigenvalues of $\Gamma$}. Let $m_i$  be the multiplicity of $\theta_i$ for $i=0,1,\ldots,D$, where $m_0=1$, then we denote \emph{spectrum} of $\GG$ by the set $\{[\theta_0]^{m_0},[\theta_1]^{m_1},\cdots,[\theta_D]^{m_D}\}$.
The matrix $A$ has all 0's on the diagonal, and thus its trace is 0, which implies that the sum of the eigenvalues (counting multiplicities) must be 0. In particular, the eigenvalue of most importance for this paper, $\theta_D$, is always negative.

 We recall that the subgraph {\it induced} on a set of vertices $W \subseteq V(\Gamma)$ is the subgraph with vertex set $W$ and edge set given by all edges in $\Gamma$ connecting two vertices in $W$.
The structure of the edges within the subgraph induced on the set $\Gamma_1(x)$ is of special importance in the theory of distance-regular graphs, and we will use the notation $\Delta(x)$ to denote the subgraph induced on $\Gamma_1(x)$, and refer to $\Delta(x)$ as the {\it local graph} at $x$.

A graph is called a {\it clique} if any two of its vertices are adjacent, and is called {\it coclique} if any two of its vertices are nonadjacent.  It is an intriguing fact that the cliques of a strongly regular or distance-regular graph have a close connection to the smallest eigenvalue of the graph, as we now describe. If $\Gamma$ is a distance-regular graph with valency $k$ and the smallest eigenvalue $\theta_{\min}$, then any clique $C$ of $\Gamma$ satisfies the inequality

\begin{equation}\label{dc-bd}
|V(C)|\leq 1+\frac{k}{|\theta_{\min}|}.
\end{equation}
This is known as the {\it Delsarte bound}, as it was proved by Delsarte in \cite{Delsarte} for strongly regular graphs, later to be generalized by Godsil in \cite{godsil-93-paper} to distance-regular graphs;  see \cite[Proposition 4.4.6]{BCN} for two short proofs of this. Equality may hold in (\ref{dc-bd}), and any clique for which it does is called a {\it Delsarte clique}. It is clear that if $\Gamma$ contains a Delsarte clique then $\theta_{\min}$ must be an integer which divides $k$, since $1+\frac{k}{|\theta_{\min}|}$ must be an integer ($\theta_{\min}$ is the root of a polynomial with integer coefficients, so it can not be a non-integer rational number).

An important concept for us will be that of a {\it geometric}
distance-regular graph; this is a distance-regular graph which contains a set of Delsarte cliques ${\cal C}$ such that every edge of $\Gamma$ lies
in a unique member of ${\cal C}$. This was first defined by Bose \cite{Bose} in relation to strongly regular graphs, later to be extended to distance-regular graphs by Godsil \cite{godsil-93-paper}.
Many of the most well-known families of distance-regular graphs are in fact geometric, including the Johnson and Hamming graphs.

If $\Gamma$ is a distance-regular graph, then for any $i \in \{1, \ldots , D\}$ its {\it distance $i$-graph} is the graph with the same vertex set as $\Gamma$, with two vertices adjacent if their distance is $i$ in $\Gamma$. $\Gamma$ is {\it bipartite} if its distance 2-graph is disconnected, and {\it antipodal} if its distance $D$-graph is disconnected. $\Gamma$ is {\it primitive} if its distance $i$-graph is connected for all $i$, and {\it imprimitive} otherwise. A remarkable theorem proved by Smith (see \cite[Theorem 4.2.1]{BCN}) shows that any imprimitive distance-regular graph with valency $k\geq 3$ is either bipartite or antipodal (or both).

An important fact about distance-regular graphs, and one we will take full advantage of, is that, loosely speaking, all small distance-regular graphs are known. To be precise,

\begin{itemize}

    \item All distance-regular graphs with $k=3$ are known (see \cite{k=3} or \cite[Theorem 7.5.1]{BCN}).

    \item All distance-regular graphs with $k=4$ are known (see \cite{brouwer1999distance}), except that there is one graph on the list (generalized hexagon of order 3) which is not yet known to be uniquely determined by its intersection array.

    \item All feasible parameter sets for strongly regular graphs with at most 512 vertices are known (see \cite[Chapter 12]{SRG}).

    \item All feasible parameter sets for primitive distance-regular graphs with diameter 3 and at most 1024 vertices are known (see \cite[Chapter  14]{BCN}).

    \item All feasible parameter sets for non-bipartite distance-regular graphs with diameter 4 and at most 4096 vertices are known (see \cite[Chapter 14]{BCN}).

    \item All feasible parameter sets for arbitrary distance-regular graphs with diameter 5 or more and at most 4096 vertices are known (see \cite[Chapter 14]{BCN}).
\end{itemize}

The last three items here are of particular interest to us, for the following reason. In many cases, we will be able to get a diameter bound on the graph in question, and from this a bound on the number of vertices. When this number is less than the relevant number above, we will often simply consult the tables in order to determine whether any graphs satisfy our requirements. Note also that bipartite graphs are not relevant to our results here, since a regular bipartite graph of valency $k$ always has $-k$ as its smallest eigenvalue; if $k \geq 4$ then this brings it below our cutoff of $-3$, and the $k=3$ case is all classified anyway. Therefore, if a graph has diameter at least $4$ and at most 4096 vertices, we know its intersection array must be found in the tables in \cite{BCN}. The same holds when the diameter $D=3$ and at most $1024$ vertices, except that we will also have to worry about the antipodal (and not bipartite) case.

\section{Linear algebraic preliminaries} \label{lin_alg_prelim}

Linear algebraic methods are prominent in the study of distance-regular graphs, and we will make heavy use of them. In this section, we collect some facts that we will need.

Eigenvalue interlacing will be important for us, in particular the following theorem, which is a simplification of the main theorem in \cite{haem} (only including the parts of it that we will use).

\begin{thm}[{cf. \cite{haem} }]\label{interlace}
	Let $A$ be a real symmetric $n \times n$ matrix with eigenvalues $\theta_1 \geq \cdots \geq \theta_n$. For some $m\leq n $, let $S$ be an $n \times m$ real matrix such that $S^\top S$ is the $m \times m$ identity matrix, and let $B = S^\top A S$ with eigenvalues $\lambda_1 \geq \cdots \geq \lambda_m$. Then the eigenvalues of B interlace those of A, that is, $\theta_i \geq \lambda_i \geq \theta_{i+n-m}$, $1 \leq i \leq m$.
\end{thm}

An {\it equitable partition} of a graph $\Gamma$ is a partition $\pi$ of the vertex set of $\Gamma$ into classes $\{P_1, P_2, \ldots , P_m\}$ such that for every pair of (not necessarily distinct) indices $i, j \in \{1, 2, \ldots , m\}$
there is a nonnegative integer $s_{i,j}$ such that each vertex $v \in P_i$ has exactly $s_{i,j}$
neighbors in $P_j$, independent of the choice of $v$. We call the matrix $B=(s_{i,j})_{1\leq i,j\leq m}$ the \emph{quotient matrix} of $\pi$. The following corollary shows that the eigenvalues of $B$ interlace the eigenvalues of $\Gamma$.

\begin{cor}[{cf.\cite{haem}}] \label{interlace-partition}
	Let $\Gamma$ be a graph on $n$ vertices with eigenvalues $\theta_1 \geq \cdots \geq \theta_n$, and suppose that $\pi$ is an equitable partition with quotient matrix $B = (s_{i,j})_{1\leq i,j\leq m}$. Let the eigenvalues of $B$ be $\lambda_1 \geq \cdots \geq \lambda_{m}$. Then, for any $i$ with $1 \leq i \leq m$, we have $\theta_i \geq \lambda_i \geq \theta_{i+n-m}$.
\end{cor}

The interlacing results lead to a number of useful facts pertaining to distance-regular graphs.

\begin{lem} \label{D4}
	For a distance-regular graph $\GG$ with diameter $D \geq 4$, we have $\theta_1 \geq a_1+1$.
\end{lem}

\begin{proof}
	
Let  $x$ be a vertex of $\GG$, and $H$ be the subgraph induced on $\{x\} \cup \GG_1(x)$, then the equitable partition of $H$ formed by the two sets $\{x\}$ and $\GG_1(x)$ leads to the quotient matrix
	
	\begin{equation}
	\begin{pmatrix}
	0  & k        \\
	1      & a_1
	\end{pmatrix}.
	\end{equation}
	
\noindent The largest eigenvalue of this matrix is $\frac{1}{2}(a_1 + \sqrt{a_1^2 + 4k})$, and thus the largest eigenvalue of $H$ is at least this large by Corollary \ref{interlace-partition}. Let now $x,y$ be two vertices with $d(x,y)=4$, and let $H'$ be the subgraph induced on $\{x\} \cup \GG_1(x) \cup \GG_1(y) \cup \{y\}$. Then $H'$ has two connected components, and each has largest eigenvalue at least $\frac{1}{2}(a_1 + \sqrt{a_1^2 + 4k}) \geq a_1+1$ by the previous argument. By Corollary \ref{interlace-partition} again, $\GG$ has two eigenvalues at least the size of $a_1+1$, in particular $\theta_1 \geq a_1+1$.
\end{proof}

A {\it tridiagonal matrix} is a square matrix which is zero for all elements other than those on the main diagonal and the diagonals immediately above and below the main diagonal. Such matrices have many interesting properties, and these will be valuable to us due the following proposition (see \cite[pp. 128-130]{BCN}).

\begin{lem} \label{subconst quot mat}

For a distance-regular graph $\Gamma$ of diameter $D$, consider the tridiagonal $(D+1) \times (D+1)$-matrix

    $$L = \begin{pmatrix}
   0  & k   &  & &   &      & \\
   1  & a_1 & b_1 & & &     & \\
      & c_2 & a_2 & b_2 & & & \\
      &     & \ddots & \ddots &\ddots & & \\
      &     &  &  &      & b_{D-1}\\
      &     &    &   & c_D &a_D
\end{pmatrix}.$$

\noindent Then the eigenvalues of $L$ are all distinct and agree with the distinct eigenvalues of $\Gamma$. Further, consider the tridiagonal $D \times D$-matrix

$$R:=\begin{pmatrix}
-1  & b_1       &     &      & & \\
1   & k-b_1-c_2 & b_2 &      & & \\
    & c_2 & k-b_2-c_3 &  & & \\
   &   &   \ddots & \ddots    &\ddots&  \\
    &     &           &      & & b_{D-1}\\
    &     &           & &c_{D-1}& k-b_{D-1}-c_D
\end{pmatrix}.$$

\noindent Then the eigenvalues of $R$ are all distinct and agree with the distinct eigenvalues of $\Gamma$ other than $k$.
\end{lem}

We will derive some consequences from these useful facts, but first we must take care of a technicality. Many powerful techniques, including eigenvalue interlacing, can be applied to symmetric matrices, but the matrices $R$ and $L$ are not symmetric. They can, however, be transformed in a simple manner into symmetric matrices; this is explained in the following lemma, which is certainly well known but is included for completeness. We omit the proof, as it is a direct calculation.

\begin{lem} \label{tridiag}
	Let $$M:=\begin{pmatrix}
	z_1  & y_1       &     &      & & \\
	x_2   & z_2 & y_2 &      & & \\
	& x_3 & z_3 &\ddots& & \\
	&     & \ddots    &\ddots& & \\
	&     &           &      & & y_{n-1}\\
	&     &           & &x_{n}& z_n
	\end{pmatrix}$$ be a tridiagonal matrix with $y_i$, $x_{i+1} > 0$ for $i = 1, \ldots , n-1$, and let $$N = diag(1,\sqrt{\frac{x_2}{y_1}}, \sqrt{\frac{x_2x_3}{y_1y_2}}, \ldots, \sqrt{\frac{x_2 x_3 \ldots x_n}{y_1y_2 \ldots y_{n-1}}}).$$ Then the matrix $\tilde M = N^{-1} MN$ takes the form
	
	$$\tilde M:=\begin{pmatrix}
	z_1  & \sqrt{x_2y_1}       &     &      & & \\
	\sqrt{x_2y_1}   & z_2 & \sqrt{x_3y_2} &      & & \\
	& \sqrt{x_3y_2} & z_3 &\ddots& & \\
	&     & \ddots    &\ddots& & \\
	&     &           &      & & \sqrt{x_ny_{n-1}}\\
	&     &           & &\sqrt{x_ny_{n-1}}& z_n
	\end{pmatrix}.$$
In particular, $\tilde M$ is symmetric, and $M$ and $\tilde{M}$ share the same spectrum.
\end{lem}

\if2
{\bf Remark:} We remark that essentially this lemma was used in the proof of Corollary \ref{interlace-partition} in \cite{haem}.\\
\fi

The following lemma can be viewed as a version of eigenvalue interlacing, but it will be easier to give a more direct proof since we require a strict inequality.

\begin{lem} \label{tri_interlace}
Let $M = \{m_{i,j}\}_{i,j=1}^n$ be a symmetric, square, tridiagonal matrix with smallest eigenvalue $\theta$. If we let $M_k =\{m_{i,j}\}_{i,j=1}^k$, for some $k < n$, with smallest eigenvalue $\theta_k$, and also have $m_{j+1,j}$ nonzero for $j = 1, \ldots, n-1$, then $\theta_k > \theta$.
\end{lem}

\begin{proof}
By the Courant minimax principle (see \cite{horn2012matrix}), we have
$$
\theta_k = \min_{v \in \RR^k \backslash \{0\}}\frac{v^\top M_k v}{v^\top v}.
$$	
Choose $v_k$ which achieves this minimum, and form a vector $\hat v \in \RR^n$ by concatenating $v_k$ with a zero vector; that is, the first $k$ entries of $\hat v$ agree with $v_k$, while the final $n-k$ entries are zeros. Again by the minimax principle,

$$
\theta \leq \frac{\hat v^\top M \hat v}{\hat v^\top \hat v} = \frac{v_k^\top M_k v_k}{v_k^\top v_k} = \theta_k.
$$

We are after a strict inequality, though, and to obtain that we may note that if we had $\theta = \theta_k$ then $\hat v$ must be an eigenvector of $M$ with eigenvalue $\theta$. However, if $j$ denotes the largest position of a nonzero component of $\hat v$ (so $1 \leq j \leq k < n$), then since $m_{j+1,j}$ is nonzero the $(j+1)$-component of $M\hat v$ would be nonzero, and thus we cannot have $M \hat v = \theta \hat v$. Thus, $\hat v$ is not an eigenvalue, and the inequality is strict.
\end{proof}

The following lemma is immediate from Lemmas \ref{subconst quot mat} through \ref{tri_interlace}, and will find use later.

\begin{lem} \label{kelhu}
With the reference to Lemma \ref{subconst quot mat},
let

    $$L_j = \begin{pmatrix}
0  & k       &  & &   &      & \\
1   & a_1 & b_1 & & &     & \\
  & c_2  &a_2 & b_2 & & & \\
   &    & \ddots& \ddots&\ddots& & \\
   &           & & &      & b_{j-1}\\
   &    &    &   & c_j &a_j
\end{pmatrix}$$
for $j < D$. If the smallest eigenvalue of this matrix is $\theta_{j}$, and the smallest eigenvalue of $\GG$ is $\theta_D$, then $\theta_j > \theta_D$. The same coclusion holds for the matrix $$R_j:=\begin{pmatrix}
-1  & b_1       &     &      & & \\
1   & k-b_1-c_2 & b_2 &      & & \\
& c_2 & k-b_2-c_3 &\ddots& & \\
&     & \ddots    &\ddots& & \\
&     &           &      & & b_{j}\\
&     &           & &c_{j}& k-b_{j}-c_{j+1}
\end{pmatrix}.$$
\end{lem}

We will now deduce some facts about distance-regular graphs from these lemmas. In \cite{-m} it was shown that, if a distance-regular graph has smallest eigenvalue at least $-m$, where $m$ is an integer and $m\geq 2$ , then $k<m(a_1+m)$. We improve this result as follows.

\begin{lem}\label{1}
	Let $\Gamma$ be a distance-regular graph with intersection array $\{k,b_1,\dots,b_{D-1};1,c_2,\dots,c_{D}\}$ and $D\geq3$. For an integer $m\geq 2$, if $\Gamma$ has smallest eigenvalue at least $-m$, then $$k<m(a_1+m)-(m-1)c_2.$$
\end{lem}

\begin{proof}

Referring to the matrix $R$ of Lemma \ref{subconst quot mat}, consider the matrix $\tilde M$ which is the symmetrization (see Lemma \ref{tridiag}) of $R + mI$. Note that the smallest eigenvalue of $\tilde M$ is nonnegative. Following Lemma \ref{tri_interlace}, if we form the matrix

	\begin{equation}
	\tilde M_2:=\begin{pmatrix}
	m-1  & \sqrt{b_1}        \\
	\sqrt{b_1}      & m+k-b_1-c_2
	\end{pmatrix}
	\end{equation}
then all eigenvalues of this matrix are strictly positive, and thus its determinant is strictly positive. Rearranging (using $k-b_1 = a_1 + 1$) yields the lemma.

\end{proof}

We remark that, since our main concern is when $m$ is $3$, this result takes the form $k < 3a_1 - 2c_2 + 9$. The following lemma will also be useful.

\begin{lem} \label{theta1diag}
	Let $\Gamma$ be a distance-regular graph with intersection array $\{k,b_1,\dots,b_{D-1};1,c_2,\dots,c_{D}\}$ and distinct eigenvalues $k=\theta_0 > \theta_1 > \cdots > \theta_D$. Then we have
	$$\theta_1 \geq \max\{-1, k-b_1-c_2, k-b_2-c_3, \ldots, k-b_{D-1}-c_D \}.$$
\end{lem}

\begin{proof}
	We first symmetrize the matrix $R$ as in Lemma \ref{tridiag}, to obtain a symmetric matrix $\tilde R$ with the same diagonal elements. We will then apply Theorem \ref{interlace} with the matrix $S$ being a $(D \times 1)$-vector with a $1$ in the $r$-th spot and zero everywhere else. The resulting matrix $B=S^\top R S$ is simply a $(1 \times 1)$-matrix whose lone entry is the $r$-th element along the diagonal of $\tilde R$, and thus of $R$. Theorem \ref{interlace} tells us that $\theta_1$ is at least as large as the eigenvalue of $B$, which is simply the $r$-th diagonal element of $R$. Applying this for each $r$ yields the result.
\end{proof}

\section{The distance-regular graphs with $c_2=1$} \label{c_2=1}

As indicated in the introduction, the geometric distance-regular graph with smallest eigenvalue $\theta_{\min} = -3$, $c_2 \geq 2$ and diameter at least 3 have been classified. This was begun by Bang \cite{gdrg-3}, who started the classification of geometric distance-regular graphs with smallest eigenvalue $-3$ and $c_2 \geq 2$, and
the results of Bang and Koolen \cite{non-exist} and Gavrilyuk and Makhnev \cite{45} completed the classification.

In this section we focus on the $c_2=1$ case. We concentrate on a non-geometric distance-regular graph $\GG$ with smallest eigenvalue $\theta_{\min}$ where $-3 \leq \theta_{\min} < -2$, $c_2=1$, and diameter at least 3. Before stating any new results, we describe the situation more carefully.

Since $c_2=1$, the local graph $\Delta$ is a disjoint union of $(a_1+1)$-cliques, and the largest cliques in $\Gamma$ are therefore of size $a_1+2$. Furthermore, each edge is contained in a unique such $(a_1+2)$-clique. As $a_1+1$ divides $k$, the quantity $t=\frac{k}{a_1+1}$ is an integer, and if we had $\theta_{\min} = -t$ then $\Gamma$ would be geometric. Since we are excluding this case, we obtain information from this observation, as we now describe.

Let $\cal{C}$ denote the set of $(a_1+2)$-cliques in $\Gamma$. Let $M$ denote the {\it vertex-clique matrix}; this is the $|V(\Gamma)| \times |\mathcal{C}|$-matrix with $M_{i,j} = 1$ if vertex $i$ is in clique $C_j$ and 0 otherwise. Each pair of adjacent vertices lies in a unique clique in $\cal{C}$, and each vertex is in $t$ such cliques. We therefore have $A = MM^\top - tI$. Since $\Gamma$ is not geometric it cannot have $-t$ as an eigenvalue, and it follows from this that the $|V(\Gamma)| \times |V(\Gamma)|$-matrix $MM^\top$ is of full rank (no zero eigenvalue), and thus that the rank of $M$ is also $|V(\Gamma)|$. This can only happen if $|\mathcal{C}| \geq |V(\Gamma)|$. We have $(a_1+2)|\mathcal{C}|=t|V(\Gamma)|$, since both quantities represent the number of 1's in $M$, and therefore $a_1 \leq t-2$. We isolate this as a lemma.

\begin{lem} \label{a1 c1=1 bound}
    Let $\Gamma$ be a non-geometric distance-regular graph with diameter $D\geq 3$ and $c_2 = 1$. Then $a_1 \leq t-2$, where $t = \frac{k}{a_1 + 1}$.
\end{lem}

The case $t=2$ (more generally, all distance-regular graphs with $k \leq 2a_1 +2$ and $D \geq 3$) has also been classified, in \cite[Theorem 16]{koo-park}.  In all cases with $t \leq 2$ we have $\theta_{\min}\geq -2$ (in fact, the case $t=2$ can only occur for line graphs by \cite[Proposition 4.3.4]{BCN}). We will therefore assume $\theta_{\min}<-2$, which forces $t \geq 3$. The following lemma characterises the possibilities which may occur.

\begin{lem}\label{5.1}
	Let $\Gamma$ be a distance-regular graph with
 smallest eigenvalue $-3\leq\theta_{\min}<-2$, $D \geq 3$, and $c_2=1$. Then $t=\frac{k}{a_1+1} \geq 3$, and the pair $$(k,a_1)\in \{(3,0),(4,0),(5,0),(6,0),(6,1),(8,1),(12,2)\}.$$
\end{lem}

\begin{proof}

Suppose first that $t=3$. By Lemma \ref{a1 c1=1 bound}, $a_1 \leq 1$. This then gives us two possibilities, $(k,a_1)\in \{(3,0),(6,1)\}$. Similarly, $t=4$ implies $a_1 \leq 2$ by Lemma \ref{a1 c1=1 bound}, and we have the possibilities $(k,a_1)\in \{(4,0),(8,1), (12,2)\}$.

Now assume $t\geq 5$. Even though the bound on $a_1$ given by Lemma \ref{a1 c1=1 bound} is growing weaker, we obtain a stronger bound now by Lemma \ref{1}, which shows $k\leq 3a_1-2c_2+8=3a_1+6$. We must therefore have $3a_1+6 \geq ta_1 + t$, and this can only occur with $t =5,6$ and $a_1 = 0$. This yields the remaining two cases, $(k,a_1)\in \{(5,0),(6,0)\}$.
\end{proof}

We will now examine each of the cases given in Lemma \ref{5.1}. Some of these cases have already been classified in other papers, while others will require a more detailed analysis. Recall that we are searching for non-geometric distance-regular graphs with $c_2 = 1$, $D \geq 3$, and $\theta_{\min} \geq -3$. We first assume $D=3$.

From Lemma \ref{5.1}, by checking the feasible intersection arrays in \cite[Chapter 14]{BCN}, we obtain the following corollary.

\begin{cor}\label{cor:c2=1;D=3}
	Let $\Gamma$ be a non-geometric distance-regular graph with $-3\leq \theta_{\min}<-2$, $c_2=1$ and diameter $3$. Then $\Gamma$ is one of the following:
\begin{enumerate}[(a)]
	\item the Odd graph $O_4$ with $\iota(\Gamma)=\{4,3,3;1,1,2\}$;
	\item the Sylvester graph with $\iota(\Gamma)=\{5,4,2;1,1,4\}$;
	\item the second subconstituent of the Hoffman-Singleton graph with $\iota(\Gamma)=\{6,5,1;1,1,6\}$;
	\item the Perkel graph with $\iota(\Gamma)=\{6,5,2;1,1,3\}$;
	\item the symplectic $7$-cover of $K_9$ with $\iota(\Gamma)=\{8,6,1;1,1,8\}$.
\end{enumerate}
\end{cor}

\begin{proof}
Since $b_1<k$, we have $k_2\leq k^2-k$ and $k_3< k^3-k^2$. So the total number of vertices $v$ is less than $k^3+1$. When $k=3,4,5,6,8$, we have $v<1024$. To obtain this same bound for $k=12$ is a bit more work, but can obtain it by applying Lemma \ref{kelhu} to the matrix

 $$\begin{pmatrix}
0  & 12       & 0  \\
1   & 2 & 9 \\
0 & 1  & a_2
\end{pmatrix}.$$

To simplify matters in this and a number of future cases, we employed a computer algebra software to handle such calculations; Mathematica notebooks containing the relevant calculations are available with the arxiv version of this paper.

The software tells us that this matrix has minimum eigenvalue less than or equal to $-3$ whenever $a_2 \leq 6$, and thus we can assume $a_2 \geq 7$ and $b_2 \leq 4$. Now
$v=1+k+k_2+k_3 \leq 1+k+kb_1+k{b_1b_2}<1024$ (as $c_3\geq c_2=1$). Therefore, in all cases, we can simply use the tables in \cite{BCN}, together with the characterization of antipodal non-bipartite graphs of diameter 3 given in the same section.

First, we look for primitive distance-regular graphs in the table of feasible intersection arrays in \cite[Chapter 14]{BCN}. We find that $\Gamma$ can be the Odd graph $O_4$, the Sylvester graph, the Perkel graph or the putative distance-regular graph with $\iota(\Gamma)=\{5,4,3;1,1,2\}$. This last one has been eliminated by Fon-Der-Flaass in \cite{Flaass1993}.
	
Next we consider the imprimitive distance-regular graph, and as mentioned before we only need to discuss the antipodal covers of distance-regular graphs with intersection array $\{k,b_1,1;1,1,k\}$. By Lemma \ref{5.1}, we know that $(k,b_1)\in\{(3,2),(4,3),(5,4),(6,5),(6,4),(8,6),$ $(12,9)\}$. We can immediately dispense with $(k,b_1) = (3,2)$ using the known classification of cubic distance-regular graphs (see \cite{k=3} or \cite[Theorem 7.5.1]{BCN}): there are none with intersection array $\{3,2,1;1,1,3\}$. The case $(k,b_1) = (6,4)$ can also be handled with known classification results, since here $a_1 = 1$ and thus any such graphs would have to be contained in the list in \cite{k6}, and since all graphs there are geometric, there are none meeting our requirements. For the others, when $(k,b_1)\neq (8,6)$, we find $a_1\neq c_2$, in this case the eigenvalues are all integers. Then we have $\theta_{\min}=-3$. Since $\theta_{\min}$ is the solution of the equation $x^2+(c_2-a_1)x-k=0$, it follows that $2k-3b_1+3=0$. So $(k,b_1)$ only can be $(6,5)$ or $(12,9)$. The first one is the second subconstituent of the Hoffman-Singleton graph with $\iota(\Gamma)=\{6,5,1;1,1,6\}$, and the second one is eliminated because the multiplicity of $-3$
is not integral by \cite[page 431]{BCN}. When $(k,b_1)= (8,6)$, we find $\Gamma$ is the Symplectic $7$-cover of $K_9$ with $\iota(\Gamma)=\{8,6,1;1,1,8\}$. This proves the corollary.
\end{proof}

Next, we discuss the $D\geq 4$ case.

\begin{cor}\label{cor:c2=1;bigD}
	Let $\Gamma$ be a non-geometric distance-regular graph with $-3\leq \theta_{\min}<-2$, $c_2=1$ and diameter $D\geq 4$. Then $\Gamma$ is one of the following:
\begin{enumerate}[(a)]
	\item the Coxeter graph with $\iota(\Gamma)=\{3,2,2,1;1,1,1,2\}$;
	\item the dodecahedron with $\iota(\Gamma)=\{3,2,1,1,1;1,1,1,2,3\}$;
	\item the Biggs-Smith graph with $\iota(\Gamma)=\{3,2,2,2,1,1,1;1,1,1,1,1,1,3\}$;
	\item the Wells graph with $\iota(\Gamma)=\{5,4,1,1;1,1,4,5\}$.
\end{enumerate}
\end{cor}

\begin{proof}

We will proceed case by case.

\gskip

{\it Case $(3,0)$}: The cubic distance-regular graphs have been classified, in \cite{k=3}. The ones which meet our conditions are the Coxeter graph, the dodecahedron, and the Biggs-Smith graph.

\gskip

{\it Case $(4,0)$}: The distance-regular graphs of valency 4 have been classified, in \cite{brouwer1999distance}. None of them meet our conditions.

\gskip

{\it Case $(5,0)$}: \if2 This and some subsequent cases require the use of computer algebra software such as Mathematica; notebooks containing the relevant calculations are available at

\gskip

https://users.monash.edu.au/gmarkow/.

\gskip
\fi
It is easy to see that if $D = 4$ then the number of vertices is less than 1000, and we can use the tables. So we need only consider $D$ large. We had the computer algebra software calculate the smallest eigenvalue of the matrix

$$\begin{pmatrix}
0  & 5       & 0 & 0 & 0 \\
1   & 0 & 4 & 0 & 0\\
0 & 1  & 4-b_2 & b_2 & 0 \\
0 & 0 & c_3  & 5-c_3-b_3 & b_3 \\
0 & 0 & 0  & c_4 & a_4
\end{pmatrix},$$
where $b_2$ ranged from 1 to 4, $b_3$ ranged from 1 to $b_2$ (because the $b_i$'s are decreasing), $c_4$ ranged from 1 to $b_3-1$, and $c_3$ ranged from 1 to $c_4$ (because the $c_i$'s are increasing). The requirement $c_4 \leq b_3 -1$ is due to \cite[Corollary 5.9.7]{BCN}, which states that if $c_{i+1} \geq b_i$ then $D \leq 3i$; in our case, if $c_4 \geq b_3$ then $D \leq 9$, which means that the number of vertices is less than $1 + 5 + 20 + 80 + 6\times 80 = 586$.
With these requirements, we found only the following partial intersection arrays
\begin{equation}
    \begin{split}
        \{5,4,3,3, \ldots;1,1,1,1, \ldots\}, \\
        \{5,4,3,2, \ldots;1,1,1,1, \ldots\}, \\
        \{5,4,2,2, \ldots;1,1,1,1, \ldots\}.
    \end{split}
\end{equation}

However, all three of these parameter sets satisfy $a_1=0, b_2 \neq b_1$, and $c_4=1$. But \cite[Corollary 4.7]{hiraki1996distance} asserts that no such distance-regular graphs can exist. Thus the only possible graphs in this case can be found in the tables.

Examining the tables gives only the Wells graph and a putative distance-regular graph with $\iota(\Gamma)=\{5,4,3,3; 1, 1, 1,2\}$, however this latter graph has been shown not to exist in \cite{fon1993distance}.

\gskip

{\it Case $(6,0)$}: Similar to the previous case, we had the computer algebra system calculate the smallest eigenvalue of the matrix

$$\begin{pmatrix}
0  & 6       & 0 & 0 & 0 \\
1   & 0 & 5 & 0 & 0\\
0 & 1  & 5-b_2 & b_2 & 0 \\
0 & 0 & c_3  & 6-c_3-b_3 & b_3 \\
0 & 0 & 0  & c_4 & a_4
\end{pmatrix},$$
where $b_2$ ranged from 1 to 5, $b_3$ ranged from 1 to $b_2$ (because the $b_i$'s are decreasing), $c_4$ ranged from 1 to $b_3-1$, and $c_3$ ranged from 1 to $c_4$ (because the $c_i$'s are increasing). The requirement $c_4 \leq b_3 -1$ is for the same reason as the previous case. With these requirements, we found only the following partial intersection array
\begin{equation}
        \{6,5,2,2, \ldots;1,1,1,1, \ldots\}.
\end{equation}

As in the previous case, we have $a_1=0, b_2 \neq b_1$, and $c_4=1$; thus no distance-regular graph can have this intersection array. Thus, again, the only possible graphs in this case can be found in the tables. However, no graphs satisfying our requirements can be found there.

\gskip

{\it Case $(6,1)$}: Fortunately for us, all distance-regular graphs with $(k, a_1) = (6, 1)$ have been classified by Hiraki, Nomura and Suzuki \cite{k6},
and it turns out that all the graphs there are geometric. Therefore there are none that meet our requirements.

\gskip

{\it Case $(8,1)$}: We begin by applying Lemma \ref{kelhu} to the matrix

 $$\begin{pmatrix}
0  & 8       & 0  \\
1   & 1 & 6 \\
0 & 1  & a_2
\end{pmatrix}.$$
The computer algebra software told us that this matrix has minimum eigenvalue less than or equal to $-3$ whenever $a_2 \leq 1$, and thus we can assume $a_2 \geq 2$. We know then that $c_3 \geq 2$ from \cite[Theorem 4.3.11]{BCN}, and thus $c_4 \geq 2$ as well. We then had the computer algebra software calculate the smallest eigenvalue of the matrix

$$\begin{pmatrix}
0  & 8       & 0 & 0 & 0 \\
1   & 1 & 6 & 0 & 0\\
0 & 1  & 7-b_2 & b_2 & 0 \\
0 & 0 & c_3  & 8-c_3-b_3 & b_3 \\
0 & 0 & 0  & c_4 & a_4
\end{pmatrix},$$
where $b_2$ ranged from 1 to 5 (because $a_2 \geq 2$), $b_3$ ranged from 1 to $b_2$ (because the $b_i$'s are decreasing), $c_4$ ranged from 2 to $b_3-1$ (for the same reason as the previous two cases), and $c_3$ ranged from 2 to $c_4$ (because the $c_i$'s are increasing). With these requirements, we found only the following partial intersection array
\begin{equation}
        \{8,6,3,3, \ldots;1,1,2,2, \ldots\}.
\end{equation}

However, in this case, by Theorem 1.1 of \cite{bang2006improving}, $c_5 > c_3 = 2$, so $c_4 \geq b_4$. By \cite[Corollary 5.9.7]{BCN}, the diameter $D \leq 12$, and thus the number of vertices is bounded above by $1+8+48+72 + 108 + 8\times 108 < 4096$. Thus the only possible graphs in this case can be found in the tables in \cite{BCN}. However, no graphs with $(k,a_1)=(8,1)$ can be found there.

\gskip

{\it Case $(12,2)$}:
There does not exist a distance-regular graph with $(k,a_1,c_2) = (12,2,1)$ and $\theta_{\min} \geq -3$, as we now show.

 We know that $c_3 \geq 2$ from \cite[Theorem 4.3.11]{BCN}. Furthermore, by Theorem 1.1 of \cite{bang2006improving} this implies that $c_5 > c_3$, so $c_5 \geq 3$, and the same theorem implies that $c_9 > c_5$, so $c_9 \geq 4$.

Let us suppose that $b_4 \leq 3$. Then, as $c_9 > b_4$, we must have that the diameter is at most 12 (see Proposition 2.4 of \cite{drgsurvey}). Using $k_i = \frac{b_0 b_1 \cdots b_{i-1}}{c_1 c_2 \cdots  c_{i}}$, we can bound $k_1=12, k_2 = 108, k_3 \leq 216, k_4 \leq 432$, and then $k_j \leq 432$ for $j=5, \ldots, 8$ since $c_5 \geq b_4$, the $c_i$'s are increasing, and the $b_i$'s are decreasing. Furthermore, for $j=9, \ldots , D$ we have $\frac{b_{j-1}}{c_j} \leq 3/4$, and thus $k_9 \leq 324$, $k_{10} \leq 243$, $k_{11} < 183$, and $k_{12} < 138$. Thus, the number of vertices of the graph is at most $1 + 12 + 108 + 216 + 5 \times 432 + 324 + 243 + 183 + 138 < 4096$, and we are done. So we may assume $b_2 = b_3 = b_4 = 4$.

Suppose now that $b_5 < c_5$. Then by Proposition 2.4 of \cite{drgsurvey} the diameter of the graph is at most 9, and we can bound as before $k_1=12, k_2 = 108, k_3 \leq 216, k_4 \leq 432, k_5 \leq 576$, and then $k_5 > k_j$ for $j = 6, 7, 8 ,9$ since $c_j > b_{j-1}$ for $j \geq 6$. We see that the number of vertices of the graph is at most $1 + 12 + 108 + 216 + 432 + 5 \times 576 < 4096$, and we are done. Thus we can assume $b_5 \geq c_5$.

We now turn again to the computer algebra software. We will apply Lemma \ref{kelhu} once more, this time to the matrix

$$\begin{pmatrix}
0  & 12 & 0 & 0 & 0 & 0  \\
1   & 2 & 9 & 0 & 0 & 0 \\
0 & 1  & 7 & 4 & 0 & 0 \\
0 & 0 & c_3 & 8 - c_3 & 4 & 0\\
0 & 0 & 0 & c_4 & 8-c_4 & 4 \\
0 & 0 & 0 & 0 & c_5 & a_5
\end{pmatrix}.$$

It is now a matter of calculating the smallest eigenvalue for each matrix of this type, as $c_3$ varies from 2 to 4, $c_4$ varies from $c_3$ to 4, $c_5$ varies from 3 to 4, and $a_5$ is equal either $8-c_5$ or $9-c_5$ (since $3 \leq c_5 \leq b_5 \leq b_4 = 4$). \if2 Any number of different softwares can perform this calculation, but the reader with access to Mathematica may copy and paste the following line:

\begin{verbatim}
    Do[Print[N@
   Min[Eigenvalues[{{0, 12, 0, 0, 0, 0}, {1, 2, 9, 0, 0, 0}, {0, 1, 7,
        4, 0, 0}, {0, 0, c3, 8 - c3, 4, 0}, {0, 0, 0, c4, 8 - c4,
       4}, {0, 0, 0, 0, c5, a5}}]]] , {c3, 2, 4}, {c4, c3, 4}, {c5, 3,
   4}, {a5, 8 - c5, 9 - c5}]
\end{verbatim}
\fi
The computer algebra system now tells us that every single one of these matrices has smallest eigenvalue strictly smaller than $-3$. Thus, again, the only possible graphs in this case can be found in the tables. However, no graphs satisfying our requirements can be found there.
\end{proof}

\section{Taylor graphs}\label{sec:Taylor}

A distance-regular graph with intersection array $\{k,c_2,1;1,c_2,k\}$ is called  \emph{Taylor graph}.
In this section we determine the Taylor graphs with smallest eigenvalue  at least $-3$.

\begin{prop}\label{prop:Taylor}
Let $\Gamma$ be a non-geometric Taylor graph with smallest eigenvalue $\theta_{\min}$. If  $-3 \leq \theta_{\min} < -2$,
then $\Gamma$ is the icosahedron, the halved 6-cube or the Gosset graph.
\end{prop}
\begin{proof}
Suppose $\Gamma$ has intersection array $\{k, c_2, 1;1, c_2, k\}$ and distinct eigenvalues $\theta_0 =k > \theta_1 > \theta_2 > \theta_3$.
The eigenvalues $\theta_1, \theta_3$ are roots of $x^2 -(a_1-c_2)x -k$. So, if $\theta_3 =-3$, then $k =  3a_1-3c_2 +9$.
This implies that $a_1=2c_2 -4$, as $k = a_1 +c_2 +1.$ By  \cite[Corollary 1.15.3]{BCN}, we know that $c_2\in\{2 , 4, 6, 10\}$ and that in each case  there is a unique graph: the 3-cube, the Johnson graph $J(6, 3)$, the halved 6-cube and the Gosset graph, respectively. Note that the first two cases are geometric, so we obtain $\Gamma$  is either the halved 6-cube or the Gosset graph, neither of which are geometric.
Now we assume that $-3 <\theta_3 <-2$. Then $a_1 = c_2$ and $\theta_3 = -\sqrt{k}$  by \cite[page 431]{BCN}. Moreover, any local graph is a connected strongly regular graph with at most 8 vertices, having no non-trivial integer eigenvalues. This means that any local graph must be the pentagon and $\Gamma$ must be the icosahedron.

This shows the proposition.
\end{proof}

\section{Terwilliger graphs}\label{sec:terwgraphs}

A Terwilliger graph is a connected non-complete graph such that, for any two vertices $u, v$ at
distance two, the subgraph induced on $\Gamma_1 (u) \cap \Gamma_1 (v)$ in $\Gamma$ is a clique of size  $\mu$ (for some fixed  $\mu \geq 0$);
clearly in the context of distance-regular graphs we have $\mu = c_2$. Note also that a Terwilliger graph does not contain induced quadrangles and then it is non-geometric from \cite[Lemma 4.2]{-m}. In this section, we determine the Terwilliger graphs satisfying $D\geq 3$.
First, for $\theta_{D}$, we give a generalized result, then we apply this result to our assumption.

\begin{prop}\label{terwgp}
	Let $\Gamma$ be a distance-regular graph with
 diameter $D\geq 3$  and smallest eigenvalue $\theta_D$. Assume that $\Gamma$ is a Terwilliger graph and $c_2 \geq 2$. If $\theta_D \geq -5$, then $\Gamma$ is one of the following:
\begin{itemize}
  \item[(1)] the icosahedron with intersection array $\{5, 2, 1; 1, 2, 5\}$ and $\theta_3 = -\sqrt{5}$;
  \item[(2)] the Doro graph with intersection array $\{10, 6, 4; 1, 2, 5\}$ and $\theta_3 = -3$ (see \cite[Proposition 12.2.2]{BCN});
  \item[(3)] the Conway-Smith graph with intersection array $\{10, 6, 4, 1; 1, 2, 6, 10\}$ and $\theta_4 = -4$ (see \cite[Section 13.2B]{BCN});
  \item[(4)]  locally the Hoffman-Singleton graph (and $c_2 =2$).
\end{itemize}
\end{prop}
\begin{proof}
	
 Assume $\theta_D \geq -5$ and $c_2\geq 2$. From Lemma \ref{1}, we have $k \leq 5a_1 - 4c_2 + 24 \leq 5a_1 + 16$.
 If $ k \leq (6+\frac{8}{57})( a_1 +1)$, we are done by \cite[Proposition 6(3)]{koo-park}. So let $5a_1 +16 \geq k > (6+\frac{8}{57})( a_1 +1)$. This means $a_1 \leq 8$ and hence $k \leq 56$.

 When $k<50(c_2-1)$, we also done by \cite[Corollary 1.16.6]{BCN}. We now assume $k \geq 50(c_2-1) \geq 50$.
 So $c_2 =2$.  This implies that, by \cite[Theorem 1.16.3]{BCN},  any local graph of $\Gamma$ is a strongly regular graph with parameters $(k, a_1, \lambda, 1)$ and $50 \leq k \leq 56$, and hence is the Hoffman-Singleton graph with parameters $(50, 7, 0, 1)$ with distinct eigenvalues $7, 2, -3$ (in this case, $\theta_D<-3$ by \cite[Corollary 3.7]{kyp}). This shows the result.
\end{proof}

If the smallest eigenvalue of $\Gamma$ is at least $-3$, we directly obtain the following corollary.

\begin{cor}\label{cor:terw}
Let $\Gamma$ be  a Terwilliger graph with diameter $D\geq 3$  and $c_2\geq2$. If the smallest eigenvalue $\theta_{D}\geq-3$, then $\Gamma$ is the icosahedron or  the Doro graph.
\end{cor}

\section{The distance-regular graphs with diameter $D\geq 3$ and $k\geq 2a_1+3,~c_2\geq 2,~\theta_D \geq -3$}\label{c2geq2}

In this section, we discuss the distance-regular graphs $\Gamma$ with diameter $D\geq 3$, valency $k\geq 2a_1+3$, intersection number $c_2\geq 2$ and smallest  eigenvalue $ \theta_D \geq -3 $. Note that by Lemma \ref{1} we have $-3\leq\theta_D<-2$ and $2a_1+3\leq k\leq 3a_1+4$, and thus $\Gamma$ must be non-geometric except for the case $k=3a_1+3$. The case when $\Gamma$ does not contain an induced quadrangle is handled in Section \ref{sec:terwgraphs}, and we will therefore assume that $\Gamma$ contains an induced quadrangle.
We will deduce in this case a diameter bound and a bound on $c_2$. We will employ the following two lemmas from \cite{BCN}.

\begin{lem}[{cf. \cite[Theorem 5.2.1, Corollary 5.2.2]{BCN}}]\label{Terw ineq}
	Let $\Gamma$ be a distance-regular graph with diameter $D$. If $\Gamma$ contains an induced quadrangle, then
	\begin{align}\label{eq8.1}
  c_i-b_i\geq c_{i-1}-b_{i-1}+a_1+2 \quad(i=1,\cdots,D). 	\end{align}
Furthermore, $$D\leq \frac{k+c_D}{a_1+2},$$
  with equality if and only if equality holds in \eqref{eq8.1} for all i.
\end{lem}

\begin{lem}[{cf. \cite[Theorem 5.2.3]{BCN}}]\label{lem5.2.3}
Let $\Gamma$ be a distance-regular graph with diameter $D\geq  \frac{k+c_D}{a_1+2} $. Then one of the following holds:
\begin{itemize}
  \item[(i) ] $\Gamma$ is a Terwilliger graph,
  \item[(ii)]  $\Gamma$  is strongly regular with smallest eigenvalue $-2$,
  \item[(iii)] $\Gamma$ is a Hamming graph, a Doob graph, a locally Petersen graph, a Johnson graph, a halfcube, or the Gosset graph.
\end{itemize}
\end{lem}
\noindent We remark that if $\Gamma$ in the above lemma has smallest eigenvalue  $-3 \leq \theta_D<-2$ and is not a Terwilliger graph (see Section \ref{sec:terwgraphs}), then $D\leq 4$; this observation will be useful in proving Proposition \ref{diambound} later.

The following lemma determines the bounds for intersection numbers $b_2$ and $c_3$.
\begin{lem}\label{7}
Let $\Gamma$ be a distance-regular graph with
diameter $D\geq3$. Assume $\Gamma$ contains an induced quadrangle and has smallest eigenvalue $\theta_{D}\geq-3$. Then $b_2\leq \frac{1}{2}b_1+\frac{1}{2}$ and $c_3\geq b_3+2c_2-2$.
\end{lem}
\begin{proof}
We have $k\leq 3a_1-2c_2+8$ by Lemma \ref{1} and $c_i-b_i\geq c_{i-1}-b_{i-1}+a_1+2$ for $i=1,2,\dots,D$ by Lemma \ref{Terw ineq}.
It follows that $b_1\leq 2a_1-2c_2+7$, and thus
\begin{align*}
	b_2 &\leq b_1+c_2-a_1-3\\
	&\leq \frac{1}{2}b_1+\frac{2a_1-2c_2+7}{2}+c_2-a_1-3\\
	&=\frac{1}{2}b_1+\frac{1}{2},
	\end{align*}
	and
	\begin{align*}
	c_3 &\geq b_3+c_2-b_2+a_1+2\\
	&\geq b_3+c_2-\frac{1}{2}b_1-\frac{1}{2}+a_1+2\\
	&\geq b_3+2c_2-2.
\end{align*}
This proves the lemma.
\end{proof}

Whether or not $\Gamma$ contains an induced $4$-claw (i.e. an induced $K_{1,4}$) will now determine our approach.
The following "claw bound" will be valuable.

\begin{lem}[{cf.\cite[Lemma 2, Proposition 3]{shilla}}]\label{5}
Let $\Gamma$ be a distance-regular graph with valency $k$ and diameter $D\geq 2$.
If $\Gamma$ contains an induced  $K_{1,t}$,  say $x, x_1, \ldots, x_t$ for $t\geq 2$, then
 $$c_2\geq \frac{t(a_1+1)-k}{\binom{t}{2}}+1.$$ Furthermore, if equality holds and  all induced paths of length $2$ lie in an induced  $K_{1,t}$ then $\Gamma$ is a Terwilliger graph.
\end{lem}

From the above lemma and  Lemma \ref{1} we obtain the following result.

\begin{lem}\label{6}
Let $\Gamma$ be a distance-regular graph with  diameter $D\geq 3$, and let $2\leq m<t$ be integers. If $\Gamma$ contains an induced $K_{1,t}$ and $\theta_{D}\geq -m$, then $$c_2\geq \frac{(t-m)(a_1+1)-m(m-1)+\binom{t}{2}+1}{\binom{t}{2}-m+1}.$$
\end{lem}

\begin{proof}
By Lemma \ref{1} we have $k\leq m(a_1+m)-(m-1)c_2-1$. On the other hand, we have $k\geq t(a_1+1)-{\binom{t}{2}}(c_2-1)$ by Lemma \ref{5}. This means $$\left(\binom{t}{2}-m+1\right)c_2\geq (t-m)(a_1+1)-m(m-1)+\binom{t}{2}+1.$$
So $c_2\geq \frac{(t-m)(a_1+1)-m(m-1)+\binom{t}{2}+1}{\binom{t}{2}-m+1}$. This proves the lemma.
\end{proof}

From Lemma \ref{5}, if $\Gamma$ contains an induced $K_{1,4}$ we immediately obtain the following lower bound for $c_2$.

\begin{cor}	\label{6'}
Let $\Gamma$ be a distance-regular graph with  diameter $D\geq 3$.
If  $\Gamma$ contains an induced $K_{1,4} $ and $\theta_D\geq -3$, then $c_2 \geq \frac{a_1 +2}{4}$.
\end{cor}

Note that, in this case, by \cite[Theorem 3.2]{gdrg-3}, we know that $\Gamma$ is non-geometric.
The following proposition gives the diameter bound.

\begin{prop}\label{diambound}
Let $\Gamma$ be a distance-regular graph with 	diameter $D\geq3$,  $\theta_D\geq -3$ and $ k\geq 2a_1+3$. If $\Gamma$ contains an induced quadrangle, then $D\leq4$.
\end{prop}
\begin{proof}
	 Note that $\Gamma$ contains an induced $K_{1,3}$ and $c_2\geq2$. If $\Gamma$ contains an induced $K_{1,3}$ but no induced $K_{1,4}$, then by \cite[Lemma 2.10(3)]{BGK} we are done. So we assume $\Gamma$ contains an induced $K_{1,4}$. Then by Corollary \ref{6'}, we have $c_2\geq \frac{a_1+2}{4}$. And by  Lemmas \ref{Terw ineq} and \ref{lem5.2.3} we have $c_i-b_i\geq c_{i-1}-b_{i-1}+a_1+2$ for $i=1,2\dots,D$ and $D< (k+c_D)/(a_1+2)$.

	Assume that $D\geq5$. Then $2k\geq(k+c_D)>5(a_1+2)$. As $c_2\geq \frac{a_1+2}{4}$ and $k\leq 3a_1-2c_2+8$, we may fix $k\leq 2\frac{1}{2}a_1+7$, and then $k \in \{2\frac{1}{2}a_1+5\frac{1}{2},2\frac{1}{2}a_1+6,2\frac{1}{2}a_1+6\frac{1}{2},2\frac{1}{2}a_1+7\}$. This implies $c_2\leq \frac{3a_1+8-k}{2}\leq\frac{a_1}{4}+\frac{5}{4}$.
 Let $t>0$ be an integer.
 If $a_1\in\{4t,4t+1,4t+2\} $, then $c_2=t+1$ and if also $a_1= 4t+3$ then $c_2=t+2$.
		
	Using $\frac{5}{2}a_1+5<k\leq3a_1-2c_2+8$ and the fact that $k\cdot a_1$ is an even integer (the handshake lemma applied to the local graph of $\Gamma$) we find the following possibilities for the quadruple $(k,a_1,b_1,c_2)$:
	\begin{align*}
		(k,a_1,b_1,c_2)\in \{&(10t+12,4t+2,6t+9,t+1),\\
		&(10t+11,4t+2,6t+8,t+1),\\
		&(10t+8,4t+1,6t+6,t+1),\\
		&(10t+6,4t,6t+5,t+1)\}.
	\end{align*}
	
    Now, $k_2=\frac{k\cdot b_1}{c_2}$ must be an integer. This doesn't help us in the $k=10t+8$ case just described, since there $c_2 \big| b_1$, but in the other three cases it severely restricts the possible values of $t$. It may be checked that the following are the only cases which may occur:
	\begin{align*}
	t\in\{5,2,1\},&\text{ if }k=10t+12,\\
	t\in\{1\},&\text{ if }k=10t+11,\\
	t\in\{3,1\},&\text{ if }k=10t+6.
	\end{align*}
	
	Let $k=10t+12$ and $t=5$. Then $k=62,~ k_2=k\cdot\frac{b_1}{c_2}=403$,
$D<\frac{2k}{a_1+2} =\frac{2\times 62}{24}<6$.  So we obtain $D=5$. By Lemma \ref{7}, we have $b_2\leq \frac{1}{2}b_1+\frac{1}{2}=20$ and $c_3\geq b_3+2c_2-2\geq 3c_2-2\geq16$.
	As a result,
	\begin{align*}
	k_3= k_2\cdot\frac{b_2}{c_3}\leq403\times\frac{20}{16}\leq503,\\
	k_4= k_3\cdot\frac{b_3}{c_4}\leq503\times\frac{20}{16}\leq628,\\
	k_5= k_4\cdot\frac{b_4}{c_5}\leq628\times\frac{20}{16}=785.
	\end{align*}
	So the order of $\Gamma$ is equal to $1+k+k_2+k_3+k_4+k_5\leq2400$. Consulting the tables in \cite[Chapter 14]{BCN}, we find none that match these requirements, and we are done in this case.
	
	The other cases for $k=10t+12$ with $t=2,1$, $k=10t+6$ and $k=10t+11$ go in same fashion, and we leave them to the reader. So we are left with the case $k=10t+8$.
	
	Let $k=10t+8$, so that $k_2=k\cdot\frac{b_1}{c_2}=6k$. We have $b_2\leq\lfloor\frac{1}{2}b_1+\frac{1}{2}\rfloor=3t+3$ and $c_3\geq 3c_2-1=3t+2$ (as $b_2<\frac{1}{2}b_1+\frac{1}{2}$). It follows that $c_3\in \{3t+2,3t+3\}$, and thus $k_3= k_2\cdot\frac{b_2}{c_3}\leq6k\cdot\frac{3t+3}{3t+2}$.
	Note that
$ b_3\leq b_2-c_2+c_3-a_1-2
	\leq3t+3-(t+1)+3t+3-4t-3
	=t+2 $
	and $c_4\geq b_4+c_3-b_3+a_1+2\geq1+3t+2-(t+2)+4t+3\geq 6t+4$.
	Thus,
	$$k_4 = k_3\cdot \frac{b_3}{c_4} \leq 6k\cdot\frac{3t+3}{3t+2} \cdot\frac{t+2}{6t+4},$$
	$$k_5 = k_4\cdot \frac{b_4}{c_5} \leq 6k\cdot \frac{3t+3}{3t+2} \cdot (\frac{t+2}{6t+4})^2.$$
	It follows that $k_4+k_5\leq 2k$ unless $t\leq2$.  If $t\leq 2$, we obtain that $v\leq 2048$ and consulting the tables \cite[Chapter 14]{BCN} allows us to close this case.

    If $t\geq 3$, then \cite[Theorem 2(2)]{Park} and \cite[Theorem 5.1.2]{BCN} imply that $\Gamma$ is an  antipodal 2-cover of a strongly regular graph (which we denote $\Delta$) of diameter 5. As $\theta_{D}\geq -3$, the graph $\Delta$ has smallest eigenvalue more than $-3$. Thus, $\Delta$ is a conference graph or $\Delta$ is a strongly regular graph with smallest eigenvalue $-2$.	
	In the first case $a_1=\frac{1}{2}k-1$ and we are done by \cite[Theorem 16]{koo-park}. In the second case we have $k \leq 28$ or $c_2\in \{2,4\}$ by \cite[Theorem 3.12.4]{BCN}. 	
	Since $c_2=t+1$, this implies that $t \leq 3$ and thus $k=10t+8\leq 38$. Note that $\Delta$ and $\Gamma$ have the same $a_1$ and $c_2$.
	As $\Gamma$ is an antipodal $2$-cover of $\Delta$, we have $|V(\Gamma)| = 2|V(\Delta)|$, so that $v= 2(1+k+k_2)\leq 534$. We finish once again by consulting \cite[Chapter 14]{BCN}.
	
 This completes the proof.
\end{proof}

In \cite[Proposition 2.11]{BGK}, it was shown that a distance-regular graph $\Gamma$ which has $k > \frac{8}{3}(a_1 +1)$, $c_2 \geq 2$, and no induced $K_{1,4}$ must be geometric with smallest eigenvalue $-3$.
On the other hand, if $\Gamma$ is a non-geometric distance-regular graph with smallest eigenvalue $\theta_D\geq -3$, and if the valency $k\geq 2a_1+3$, then $c_2\geq \frac{a_1-5}{6}$ by \cite[Proposition 9.9]{drgsurvey}. We will improve this lower bound on $c_2$ later.

The following results about geometric distance-regular graphs will be required;
they are due to Metsch (see \cite{metsch1995characterization} or \cite[Proposition 9.1]{drgsurvey}).

\begin{prop} \label{prop5}
	Let $\Gamma$ be an distance-regular graph with valency $k$ and diameter $D\geq 2$.
	Assume that the following two properties hold for some positive integer $m\geq 1$:
	
\begin{enumerate}
\item[(1)] $a_1 \geq (2m-1)(c_2-1)$,
\item[(2)] $k < (m+1)(a_1+1)-\frac{1}{2}m(m+1)(c_2-1)$.
\end{enumerate}
Define a line to be a maximal clique $C$ satisfying $|V(C)| \geq a_1 + 2 - (m-1)(c_2-1)$.
Then every vertex is on at most $m$ lines and any edge lies in a unique line.
\end{prop}

The following proposition was shown by Van Dam, Koolen and Tanaka, see \cite[Proposition 9.8]{drgsurvey}.
\begin{prop}\label{propDKT}
	Let $\Gamma$ be a distance-regular graph with diameter $D$ with the property that there exists a positive
	integer $m$ and a set $\mathcal C$ of cliques in $\Gamma$ such that every edge is contained in exactly one clique of $\mathcal C$ and every
	vertex $x$ is contained in exactly $m$ cliques of $\mathcal C$. If $|{\mathcal C}| < |V(\Gamma)|$, then $\Gamma$ is geometric with smallest eigenvalue $-m$.
	In particular, this is the case if $\min\{|V(C)| \mid C \in {\mathcal C}\} > m$.
\end{prop}

As a consequence of Propositions \ref{prop5} and \ref{propDKT} we have:

\begin{thm} \label{thm7.6}
	Let $\Gamma$ be a distance-regular graph with smallest eigenvalue $ \theta_{\min} \geq -3$, $c_2 \geq 2$ and $k \geq 2a_1+3$. If $a_1 \geq 5(c_2-1)$ and $k < 4(a_1+1)-6(c_2-1)$, then $\Gamma$ is geometric with $\theta_{\min}=-3$.
\end{thm}
\begin{proof}
	By Proposition \ref{prop5}, every vertex lies on at most $3$ lines and any edge lies on a unique line, where a line is a maximal clique $C$
	satisfying $|V(C)| \geq a_1 + 2 - 2(c_2-1)$. Note that any line has at least $3(c_2-1) +2 \geq 5$ vertices. As $k \geq 2a_1 +3$, we see that every vertex lies on exactly $3$ lines as a line can not have more than $a_1 +2$ vertices. Hence we are done by Proposition \ref{propDKT}.
\end{proof}

Combining the conclusions above, we obtain a the following useful corollary.

 \begin{cor} \label{fu}
	Let $\GG$ be a distance-regular graph that has an induced quadrangle but no $K_{1,4}$'s. If $\theta_{\min} \geq -3$,   $k\geq 2a_1+3$ and
  $ c_2 < \frac{a_1+6}{5}$, then $\GG$ is geometric.
\end{cor}
\begin{proof}
It is easy to verify that $\GG$ has an induced $K_{1,3}$ and $c_2\geq 2$.
Assume that $\GG$ is not geometric. If $\GG$ has no induced $K_{1,4}$, as mentioned earlier, we obtain that $k \leq \frac{8}{3}(a_1+1)$ by \cite[Proposition 2.11]{BGK}. Then $4(a_1+1)-6(c_2-1) \geq (4-\frac{6}{5})(a_1+1)+\frac{6}{5} > \frac{8}{3}(a_1+1) \geq k$, however this contradicts Theorem \ref{thm7.6}.
\end{proof}

It is known by Corollary \ref{6'} that if a distance-regular graph $\GG$ has an induced $K_{1,4}$, then $c_2\geq \frac{a_1+2}{4}$. In fact, this can be improved in the case that $\GG$ is not geometric and has induced quadrangles. The following proposition asserts that, under these assumptions, the case $\frac{a_1+2}{4}\leq c_2 < \frac{a_1+6}{5}$ does not occur.

\begin{prop}\label{propc_2}
Let $\GG$ be a distance-regular graph with $D\geq 3$, $\theta_D \geq -3$, $c_2 \geq 2$ and $k\geq 2a_1+3$.
	If  $\GG$ is   non-geometric and contains  an induced  quadrangle, then $ c_2 \geq \frac{a_1+6}{5} $.
\end{prop}
\begin{proof}
    Assume that $\GG$ is not geometric.
If $c_2\leq\frac{1}{5}(a_1+1)$, we have $4(a_1+1)-6(c_2-1)\leq k\leq 3a_1-2c_2+8 $ by Theorem \ref{thm7.6} and Lemma \ref{1}, which implies that $c_2\geq \frac{a_1+2}{4}$. This shows that $5\leq a_1\leq 10$.  Because $c_2$ is an integer, we may assume $\lceil \frac{a_1+2}{4}\rceil \leq c_2\leq \lfloor \frac{a_1+5}{5}\rfloor$, and then the possibilities for the pair $(a_1,c_2)$ are $(5,2),~(6,2),~(10,3)$.

Assume now that $\GG$ has an induced $K_{1,4}$ (otherwise we may simply apply Corollary \ref{fu}). For the case  $(a_1,c_2)=(5,2)$, we have $k=18=3(a_1+1)$, and then $\GG$ has no an induced $K_{1,4}$ by \cite[Lemma 4.3]{BGK}, a contradiction.
For the other cases, because $\GG$ contains an induced quadrangle we obtain $D\leq 4$ by Proposition \ref{diambound}. If  $(a_1,c_2)=(6,2)$, then $k=22$, $k_2=k\cdot\frac{b_1}{c_2}=165$, and $k_3=k_2\cdot\frac{b_2}{c_3}\leq  k_2\cdot\frac{\frac12(b_1+1)}{(2c_2-2)}= 660$ by Lemma \ref{7}. So $v\leq 1+k+k_2+k_3+k_4 \leq 2(1+k+k_2+k_3)<2048$. We may therefore check the feasible array table in \cite[Chapter 14]{BCN}, and we find that this parameter set does not exist. The case  $(a_1,c_2)=(10,3)$ can be shown to be non-existent in the same way.
\end{proof}

\section{Bound for $a_1$} \label{sec:a1bound}
In this section, we suppose that the  distance-regular graph $\GG$  has the same constraints as in Section \ref{c2geq2}. Suppose further that $\GG$ is non-geometric and contains an induced quadrangle. In order to search for feasible arrays for $\GG$ with the aid of a computer, we will find an upper bound for $a_1$. We have already reduced to the case $2a_1+3\leq k\leq 3a_1-2c_2+8$. By  Proposition \ref{propc_2}, we have $c_2\geq \frac{a_1+6}{5}$ and then $k\leq\frac{13a_1+28}{5}$.
We will assume $a_1 \geq 44$, so that $c_2\geq10$ and $k\leq\frac{13a_1+28}{5}\leq \frac{8}{3}(a_1+1)$; since we are looking for an upper bound on $a_1$, and this upper bound will be greater than 44, this assumption on $a_1$ is no restriction. The following lemma will be crucial for us.

\begin{lem} \label{bigguy}
	Let $\Gamma$ be a distance-regular with $v$ vertices, valency $k$, diameter $D\geq3$ and distinct eigenvalue $k=\theta_0>\theta_1>\dots>\theta_D$ with respective multiplicities $m_0=1,m_1,\dots,m_D$. Let $\zeta=\frac{v-1}{k},~\gamma=\frac{\theta_1}{a_1+1},~\delta=\frac{k}{a_1+1}$. Assume there exist an $N\leq4$ such that $\frac{k}{N}\leq \min\{m_1,m_D\}$. Then $a_1+1\leq \frac{N(\zeta-1)\delta}{\gamma^2}$.
\end{lem}
\begin{proof}
	We have $vk=\sum_{i=0}^D m_i\theta_i^2\geq k^2+m_1\theta_1^2+m_D\theta_D^2$. As $\Gamma$ is regular and non-complete, we find $\theta_D\leq-2$. So we obtain
	\begin{align*}
	(\zeta k+1)\cdot k&\geq k^2+\frac{k}{N}\theta_1^2+\frac{k}{N}\theta_D^2\\
	&\geq k^2+\frac{k}{N}\theta_1^2+\frac{k}{N}\cdot4
	\end{align*}
	Therefore $\theta_1^2\leq N(\zeta-1)k$.
	
	By the definition of $\gamma,\delta$ we find that $a_1+1\leq \frac{N(\zeta-1)\delta}{\gamma^2}$.
\end{proof}

We will obtain an upper bound for $a_1$ according to the above Lemma in this section. Let us start with discussing parameter $N$. Suppose the distance-regular graph $\Gamma$ has  distinct eigenvalues $k=\theta_0>\theta_1>\cdots>\theta_D$ with respective multiplicities $m_0=1,m_1,\cdots, m_D$. By \cite[Theorem 3.6]{kyp} and \cite[Theorem 4.4.3(ii)]{BCN}, we only need to discuss the cases that the second largest eigenvalue $\theta_1$ satisfies $\frac{b_1}{2}-1<\theta_1\leq b_1-1$ (as $\theta_D\geq -3$). Note that the distance-regular graphs with $\theta_1=b_1-1$  were classified in \cite[Theorem 4.4.11]{BCN}. Therefore, we assume that $\frac{b_1}{2}-1<\theta_1< b_1-1$ in the following discussion.

When $m_1\geq k$, we can assume $N\leq 1$, which satisfies the condition of  Lemma \ref{bigguy}. When $m_1< k$, we have known from \cite[Theorem 4.4.4]{BCN} that
if $\theta_1\notin\mathbb{Z}$, $\theta_{1}$ and $\theta_D$ are conjugate  algebraic integers,   hence $m_1=m_D$; if  $\theta_1\in \mathbb Z $, $\theta_1$ divides $b_1-1$  which is not in the scope of our discussion. So we will discuss $N$ when $m_1<k$ and $\theta_1\notin\mathbb{Z}$.

In this case, it  follows that $m_D=m_1<k$ and then the local graph of $\Gamma$ has an eigenvalue  $-\frac{b_1}{\theta_D+1}-1$   with multiplicity at least $k-m_D$ by \cite[Theorem 4.4.4]{BCN}. Then we have
\begin{align*}
 & ka_1\geq  a_1^2+(k-m_D)(-\frac{b_1}{\theta_D+1}-1)^2, \\
\implies~~ & a_1(b_1+1) \geq  (k-m_D)(\frac{b_1}{2}-1)^2.
\end{align*}
Since $k\geq2a_1+3$, we have $a_1\leq b_1-2$. It follows that $(b_1-2)(b_1+1)\geq \frac{1}{4}(k-m_D)(b_1-2)^2$. One can easily  check that $m_D\geq k-4$, so $m_1\geq k-4\geq 0.956k$ as $k\geq 2a_1+3\geq91$. Thus, we have $N\leq \frac{1}{0.956}$. Note that this upper bound value  is greater than 1, so we always assume $N\leq  \frac{1}{0.956}$ whether $m_1\geq k$ or $m_1<k$.

The next two propositions utilize Lemma \ref{bigguy} to provide upper bounds for $a_1$ when $\Gamma$ has diameter $3$ or $4$.

\begin{prop}\label{a_1 bound3}
	Let $\Gamma$ be a non-geometric distance-regular graph with $D=3, ~\theta_D\geq -3$ and $k\geq 2a_1+3\geq 91$.
 If $\Gamma$ contains an induced quadrangle, then $a_1<100$.
\end{prop}

\begin{proof}
 We first mention that if  $a_1\geq44$, then $c_2\geq \frac{a_1+6}{5}$ and $k\leq\frac{13a_1+28}{5}\leq \frac{8}{3}(a_1+1)$.
 It is  easy to see that $k_2=k\cdot \frac{b_1}{c_2}\leq k\cdot\frac{\frac{1}{5}(13a_1+28)-(a_1+1)}{\frac{1}{5}(a_1+6)}<k\cdot\frac{8(a_1+3)}{a_1+6}<8k$.

When $\gamma^2\geq1$,  by Lemma \ref{7} we have
 $k_3=k_2\cdot\frac{b_2}{c_3}\leq k_2\cdot\frac{\frac{1}{2}(b_1+1)}{\frac{2}{5}(a_1+1)}\leq k_2\cdot\frac{8a_1+28}{4(a_1+1)}<17k$,
and hence $\zeta-1=\frac{k_2+k_3}{k}<25$.
Also $N\leq \frac{1}{0.956}$ and $\delta\leq \frac 83$, and thus
by Lemma \ref{bigguy} we obtain $a_1< 69$.

When $\gamma^2<1$,  we have  $\min\{a_1+1,a_3\}\leq \theta_1<a_1+1$ from  \cite[Lemma 6]{shilla}. This  implies
$a_3\leq a_1$ and then $c_3\geq b_1+1$. So we have $k_3=k_2\cdot\frac{b_2}{c_3}\leq k_2\cdot\frac{\frac12(b_1+1)}{c_3} \leq \frac{1}{2}k_2$.
Note also that $\gamma^2\geq\frac{1}{5}$, since otherwise we have $\frac{b_1}{2}-1<\theta_1<\frac{\sqrt{5}}{5}(a_1+1)<\frac{9}{20}(a_1+1)$, and then we obtain $k<1.9a_1+3.9<2a_1+3$ as $a_1\geq 44$, which contradicts our assumptions. We have therefore reduced to the case $\frac 15\leq \gamma^2<1$.

If $\frac{1}{3}\leq\gamma^2<1$, we see that  $k_2<8k,~k_3\leq \frac 12 k_2<4k$. Hence $\zeta-1=\frac{k_2+k_3}{k} <12$. By Lemma \ref{bigguy} and $\delta\leq \frac 83$, it follows that $a_1<100$.

If $\frac{1}{4}\leq \gamma^2<\frac{1}{3}$, it follows that $\theta_1<\frac{\sqrt{3}}{3}(a_1+1)<\frac35(a_1+1)$. As $\frac{b_1}{2}-1<\theta_1$, we obtain that $b_1<\frac{6a_1+16}{5}<\frac{5}{4}(a_1+1)$. Hence $ k<\frac{9}{4}(a_1+1)$, and then $ \delta< \frac{9}{4}$. In this case we have $k_2=k\cdot \frac{b_1}{c_2}< k\cdot\frac{\frac{1}{5}(6a_1+16)}{\frac{1}{5}(a_1+6)}<6k$ and
 $k_3\leq  \frac 12 k_2  <3k$, so $\zeta-1=\frac{k_2+k_3}{k} <9$  and thus $a_1<84$.

If $\frac 15\leq \gamma^2<\frac{1}{4}$, we find that $\frac{b_1}{2}-1<\theta_1<\frac12(a_1+1)$. This implies that $k< 2a_1+4$, and so $k=2a_1+3$.  By \cite[Proposition 9.5]{drgsurvey} we have $c_2\geq \frac{3(a_1+1)-(2a_1+3)}{3}+1>\frac{a_1+2}{3}$, and then we obtain $k_2=k\cdot \frac{b_1}{c_2}<k\cdot\frac{a_1+2}{\frac13(a_1+2)}=3k$ and $k_3\leq \frac12k_2<1.5 k$. Hence $\zeta-1<4.5$. Also $\delta<2.03$ as $a_1\geq 44$. By Lemma \ref{bigguy} it follows that $a_1<47$.

This proves the proposition.
\end{proof}

\begin{prop}\label{a_1 bound4}
Let $\Gamma$ be a non-geometric distance-regular graph with $D=4, ~\theta_D\geq -3$ and $k\geq 2a_1+3\geq 91$.	
 If $\Gamma$ contains an induced quadrangle, then $a_1<100$.
\end{prop}

\begin{proof}
  The proof is similar to that of Proposition  \ref{a_1 bound3}.  If $a_1\geq 44$,  we obtain that $c_2\geq \frac{a_1+6}{5}$ and $k\leq\frac{13a_1+28}{5}\leq \frac{8}{3}(a_1+1)$.
Also $b_1=k-a_1-1\leq \frac{8a_1+23}{5}$, $b_2\leq \frac12 (b_1+1)\leq \frac{4a_1+14}{5}$, $c_3\geq b_3+2c_2-2\geq c_1+2c_2-2\geq \frac{2a_1+7}{5}$ and $c_4\geq c_3-b_3+a_1+2\geq 2c_2+a_1\geq \frac{7a_1+12}{5}$ by Lemma \ref{Terw ineq} and Lemma \ref{7}.

 It is easy to show that $k_2=k\cdot\frac{b_1}{c_2}<8k$,
 $k_3=k_2\cdot\frac{b_2}{c_3}\leq 2k_2<16k $ and $k_4=k_3\cdot\frac{b_3}{c_4}\leq k_3\cdot\frac{b_2}{c_4}\leq k_3\cdot\frac{4a_1+14}{7a_1+12}< 10k$. Hence
 $\zeta-1=\frac{k_2+k_3+k_4}{k}<34$.
 Note that when $D=4$ we have $\gamma^2\geq 1$ by Lemma \ref{D4}. Also $N\leq \frac{1}{0.956}$ and $\delta\leq \frac 83$,  and thus
by Lemma \ref{bigguy} we obtain $a_1< 94$.
\end{proof}

\section{Computational results}\label{sec:comput}
In this section, with the aid of computer, we will find all the feasible arrays for  non-geometric distance-regular graphs with an induced quadrangle, smallest eigenvalue at least $-3$, and diameter $3$ or $4$ whose $a_1$ satisfies the bounds in Propositions \ref{a_1 bound3} and \ref{a_1 bound4}. To be precise, by "feasible" we mean that the intersection array $\{k=b_0, b_1, \ldots , b_{D-1};1=c_1,c_2, \ldots , c_D\}$ satisfies the following criteria.

\begin{enumerate}
    \item $1 \leq a_1 < 100$ (see Section \ref{sec:a1bound}).
    \item $2a_1+3 \leq k \leq 3a_1 -2c_2 + 8$ (see Section \ref{c2geq2}).
    \item $k_i = \frac{b_0 b_1 \cdots b_{i-1}}{c_1 c_2 \cdots  c_{i}}$ are all integers, and $k_i a_i$ is even for $i \in \{1, \ldots, D\}$ (this is the handshake lemma applied to the subgraphs induced on $\Gamma_i(x)$).
    \item $\min\{\frac{a_1+6}{5},\frac{a_1+2}{4}\} \leq c_2 \leq \frac{3a_1 + 8-k}{2}$ (see Section \ref{c2geq2}).
    \item The $b_i$'s are nonincreasing and $c_i$'s are nondecreasing, and all are between $1$ and $k$. Furthermore, $b_i \geq c_{D-i}$ for all $i \in \{1, \ldots, D\}$ (see Section \ref{prelim}).
    \item For all $i \in \{1, \ldots, D\}$, we have $c_i-b_i \geq c_{i-1} - b_{i-1} + a_1 + 2$ (this is an inequality due to Terwilliger; see
  Lemma \ref{Terw ineq}).
    \item $2c_2 - 1 \leq c_3$. (this is a consequence of Lemma \ref{7}).
    \item $(3a_1+9-k)(a_2+3) - 3b_1c_2 \geq 0$ (this follows by applying Lemma \ref{tri_interlace} to the $3\times3$ principal submatrix of the matrix $L+3I$, where $L$ is defined in Section \ref{lin_alg_prelim}).
    \item For $D=3$, we have $\frac{b_1}{2}-1<\theta_1<b_1-1$ (this lower bound for $\theta_1$ is due to the main result in \cite{kyp}, while the upper bound is due to \cite[Theorem 4.4.3(ii) and Theorem 4.4.11]{BCN}, using the fact that we have an induced quadrangle).
    \item The multiplicity $m_i$ of eigenvalue $\theta_i$ is a  positive integer for $i=0,1,\cdots,D$.
    \item For $D=3$, we have that at least one of the non-trivial eigenvalues is an integer (this is because roots to the characteristic polynomials must occur in algebraic conjugate pairs).
    \item Any pair of eigenvalues which are algebraic conjugates must have the same multiplicity.
\end{enumerate}

The python code used to determine these intersection arrays is available with the arxiv submission of this paper.

\subsection{Diameter 3}\label{sec:compD3}

With these requirements, the code found the following feasible intersection arrays:
\begin{enumerate}[(1)]
        \item $\{6,4,1;1,1,6\}$;
        \item $\{7,4,1;1,2,7\}$;
	\item $\{9,6,1;1,2,9\}$;
	\item $\{15,8,1;1,4,15\}$;
	\item $\{15,10,1;1,2,15\}$;
	\item $\{18,12,1;1,2,18\}$;
	\item $\{27,16,1;1,4,27\}$;
	\item $\{39,24,1;1,4,39\}$;
        \item $\{45,26,3;1,6,39\}$;
        \item $\{45,24,1;1,8,45\}$;
        \item $\{45,24,2;1,10,36\}$;
        \item $\{51,30,1;1,6,51\}$;
        \item $\{60,35,9;1,6,42\}$;
        \item $\{87,48,1;1,12,87\}$;
        \item $\{207,120,1;1,20,207\}$.

\end{enumerate}

Array $\SSII$ has $c_2=1$, and has therefore been eliminated in Section \ref{c_2=1}.

For array $\SSII$, we know it is the Klein graph and it fits in with our requirements.

For array $\SSII$, we find this array is the case $\{st,s(t-1),1;1,t-1,st\}$ when $s=3,t=3$. By \cite[Proposition 12.5.2]{BCN}, we know this graph derived from a strongly regular graph of parameters $(40,12,2,4)$ with a partition $\mathfrak{S}$ of its point set into $4$-cliques. By \cite{spence2000}, we know there are $28$ non-isomorphic strongly regular graphs with parameters $(40,12,2,4)$. Among those cases we
know $GQ(3,3)$ minus a unique spread is eliminated since it is geometric. But there is also a strongly regular graph that is not geometric and it has 4 different spreads. The spreads come in pairs and give two non-isomorphic distance-regular graphs that are non-geometric.

For $\SSII$, the spectrum is $\{[15]^{1}, [5]^{12}, [-1]^{15}, [-3]^{20}\}$. Since the multiplicity of $5$ is $12<15$ and $\frac{b_1}{\theta_1+1}=\frac{8}{6}$ is not integral, by \cite[Theorem 4.4.4]{BCN} it follows that this graph does not exist.

For array $\SSII$, we find this array is the case $\{st,s(t-1),1;1,t-1,st\}$ when $s=5,t=3$. Similarly, we know this graph derived from a strongly regular graph of parameters $(96,20,4,4)$ with a partition $\mathfrak{S}$ of its point set into $6$-cliques. By \cite[Section 7.2]{Brouwer2003}, we know that they are $GQ(5,3)$ minus a spread and an antipodal $6$-cover of $K_{16}$. The first one is geometric and the second one is not. From that paper, we also know there exists a non-geometric distance-regular graph derived from the collinearity graph of $GQ(5,3)$.

For $\SSII$, we find this array is the case $\{st,s(t-1),1;1,t-1,st\}$ when $s=6,t=3$. We know this graph derived from a strongly regular graph of parameters $(133,24,5,4)$ with a partition $\mathfrak{S}$ of its point set into $7$-cliques. By \cite{DSZ1976} we know that $GQ(6,3)$ does not exist. However, since the classification of strongly regular graph with parameters $(133,24,5,4)$ is still pending, we do not know if there exist some non-geometric cases with this intersection array by now.

For array $\SSII$, we find the spectrum is $\{[27]^{1}, [9]^{28}, [-1]^{27}, [-3]^{84}\}$ and $a_3=0$, then it only can be some distance-regular antipodal $5$-cover. Since the multiplicity of $9$ is $28<27+5-2$ and $\frac{b_1}{\theta_1+1}=\frac{16}{10}$ is not integral, by \cite[Theorem 3]{GK1995} it follows that this graph does not exist.

For array $\SSII$, the graph in question has been eliminated by \cite[Proposition 5.20]{BGK}.

Array $\SSII$ can be eliminated, as follows. The corresponding graph, if it existed, would have $1+45+195+15=256$ vertices. It would not be bipartite, since $-45$ is not an eigenvalue, and it is also not antipodal because $a_3 > 0$. As such, the graph would be primitive, and therefore its intersection array would be found in the tables in \cite[Chapter 14]{BCN}; however this array does not appear there, and therefore the graph does not exist.

We can also eliminate array $\SSII$, as follows. The spectrum of a graph with this intersection array can be calculated as $\{[45]^1, [15]^{23}, [-1]^{45},[-3]^{115}\} $. Since the multiplicity of the eigenvalue $\theta_1=15$ is less than $k=45$, by \cite[Theorem 4.4.4]{BCN} we must have $\theta_1+1=16$ dividing $b_1=24$, a contradiction. Thus no graph with this intersection array can exist.

The remaining cases $\SSII$-$(15)$ can all be eliminated in the same manner as case $(10)$. The following table lists the intersection arrays and their corresponding spectra. The reader may check that, in each case, $\theta_1$ has multiplicity less than $k$, and that $\theta_1+1$ does not divide $b_1$. Therefore no graphs with these arrays can exist by \cite[Theorem 4.4.4]{BCN}.

\begin{table}[h!]
\begin{tabular}{l|l}
Intersection array & Spectrum\\
\hline
$\{45,24,2;1,10,36\}$ & $\{[45]^1, [15]^{16}, [5]^{18},[-3]^{125}\}$ \\
$\{51,30,1;1,6,51\}$ &  $\{[51]^1, [17]^{39}, [-1]^{51}, [-3]^{221}\}$\\
$\{60,35,9;1,6,42\}$ &  $\{[60]^1, [24]^{35}, [6]^{50}, [-3]^{400}\}$ \\
$\{87,48,1;1,12,87\}$ & $\{[87]^1, [29]^{33}, [-1]^{87}, [-3]^{319}\}$ \\
$\{207,120,1;1,20,207\}$ & $\{[207]^1, [69]^{52}, [-1]^{207}, [-3]^{1196}\}$
\end{tabular}
\end{table}

From the discussion above, we can now summarize the computational results.

\begin{prop}\label{9.1}
Let $\Gamma$ be a non-geometric distance-regular graph with smallest eigenvalue $-3\leq\theta_{\min}<-2, c_2\geq2, D= 3$ and contains an induced quadrangle, then $\Gamma$ is one of the following:
	\begin{itemize}
		\item the Klein graph with $\iota(\Gamma)=\{7,4,1;1,2,7\}$;
		\item a distance-regular graph with $\iota(\Gamma)=\{9,6,1;1,2,9\}$;
		\item a  distance-regular graph with $\iota(\Gamma)=\{15,10,1;1,2,15\}$;
		\item a putative distance-regular graph with $\iota(\Gamma)=\{18,12,1;1,2,18\}$.
	\end{itemize}
\end{prop}

\subsection{Diameter 4}\label{sec:compD4}

Remarkably, our code found no intersection arrays that fit our requirements in the diameter 4 case.

\section{Proof of Theorem \ref{main}}\label{sec:proof}
In this section, we will collect our results and present the proof of the main theorem.

\begin{proof}
Let $\Gamma$ be a non-geometric distance-regular graph with diameter at least 3 and the smallest eigenvalue $\theta_{\min}$ satisfies
$-3 \leq \theta_{\min}<-2$.
First we consider the case $c_2 =1$. They are classified in Corollary \ref{cor:c2=1;D=3} and Corollary \ref{cor:c2=1;bigD} and we obtain  the graphs of $(a)-(i)$.
So now we may assume $c_2 \geq 2$. If $\Gamma$ has no induced quadrangle then they are classified in Corollary \ref{cor:terw} and we obtain the graphs $(j)-(k)$.
From now on we  may assume that $\Gamma$ contains an induced quadrangle.
Koolen and Park \cite{koo-park} showed that a distance-regular graph with diameter at least 3 and $k \leq 2a_1 +2$ is a line graph, a Taylor graph,
the Johnson graph $J(7,3)$ or the halved 7-cube.
The Johnson graph $J(7,3)$ is geometric, line graphs have smallest eigenvalue at least $-2$ and the non-geometric Taylor graphs are classified in Proposition \ref{prop:Taylor}.  From this we obtain the graphs $(l)-(n)$. So we may assume from now that $k \geq 2a_1 +3$ and $k \leq 3a_1- 2c_2+8\leq 3a_1 +4$.
In Proposition \ref{diambound} we showed that $\Gamma$ must have diameter at most 4. In Propositions  \ref{a_1 bound3} and
 \ref{a_1 bound4} we obtain that such a graph must have $a_1 <100$. Then we found in Sections \ref{sec:compD3} and \ref{sec:compD4} the remaining intersection arrays except for the graph $(q)$ that satisfies $\theta_1=b_1-1$. The distance-regular graphs with $\theta_1=b_1-1$  were classified in \cite[Theorem 4.4.11]{BCN} and the Doob graph with intersection array $\{9,6,3;1,2,3\}$, that is  the graph $(q)$,  is also within the scope of our discussion.  This finishes the proof.
\end{proof}

\section{Acknowledgements}

The authors would like to thank Edward Spence and Junming Wang for valuable conversations. J.H. Koolen is partially supported by the National Key R. and D. Program of China (No. 2020YFA0713100),
the National Natural Science Foundation of China (No. 12071454 and No. 12371339), and the Anhui Initiative in Quantum Information Technologies (No. AHY150000). Xiaoye Liang is partially supported by the National Natural Science Foundation of China (No. 12201008 and No. 12371339), the Foundation of Anhui Jianzhu University (No. 2022QDZ18), the Natural Science Research Project of Anhui Educational Committee (No. 2023AH050194), the Innovation Team of Operation Research  and Combinatorial Optimization of Anhui Province (No. 2023AH010020).

\bibliographystyle{plain}
\bibliography{ref}
\end{document}